\documentclass[12pt,a4paper]{article}
\usepackage[utf8]{inputenc}
\usepackage[english]{babel}
\usepackage{amsmath,amsfonts,amssymb,amsthm,mathrsfs,authblk,enumerate,cite,mathtools,setspace,relsize}
\usepackage{amsfonts}
\usepackage{times}
\usepackage{layout}
\usepackage[left=3cm, right=3cm, top=3cm,bottom=3cm]{geometry}
\newtheorem{thm}{Theorem}[section]
\newtheorem{defn}{Definition}[section]
\newtheorem{lem}{Lemma}[section]
\newtheorem{prop}{Proposition}[section]
\numberwithin{equation}{section}
\theoremstyle{plain}

\theoremstyle{definition}

\title{On Hilfer-Prabhakar fractional derivatives Sawi transform and its applications to fractional differential equations}
\author{MOHD KHALID, SUBHASH ALHA }
\affil{Department of Mathematics, Maulana Azad National Urdu University,\\Gachibowli,Hyderabad-500032, India.\\ 

\vspace{3mm} \textup{Email: khalid.jmi47@gmail.com, subhashalha@manuu.edu.in}}

\date{}

\begin{document}

\maketitle

\noindent\textbf{Abstract:} The goal of this paper is to study the Sawi transform and its relationship to Hilfer-Prabhakar and regularized Hilfer-Prabhakar fractional derivatives, as well as to present some lemmas related to the Sawi transform. Additionally, the paper aims to find solutions for Cauchy type fractional differential equations using Hilfer-Prabhakar fractional derivatives, by utilizing the Sawi and Fourier transforms, and involving the three-parameter Mittag-Leffler function.

\vspace{3mm} \noindent \textbf{2010 Mathematical Subject Classification :} 44A35, 26A33, 42B10.\\
----------------------------------------------------------------------------------

\section{Introduction} Fractional calculus is a field of mathematics that deals with the study of fractional integrals and derivatives of real or complex orders. There are various types of fractional integrals and derivatives used in literature. In recent years, fractional calculus has gained significant attention from researchers due to its wide range of applications in various fields such as physics, chemistry, biology, control theory, economics, electronics and other areas of science and technology \cite{ Mainardi2010},\cite{Miller1993},\cite{Oldham1974},\cite{prabhakar1971singular},\cite{Samkokilbas1993},\cite{Agrawal2008omp}. The Prabhakar integral is a modification of the Riemann-Liouville integral that utilizes a kernel involving the three-parameter Mittag-Leffler function. The Hilfer-Prabhakar derivative and its regularized version were first introduced in a research paper in \cite{regularizedprabhakar}.

Many researchers have used Hilfer-Prabhakar fractional derivatives in various fields due to their special properties, which include the ability to combine several integral transforms such as Laplace, Fourier, Sumudu, Shehu, Elzaki and others.
In the research paper by Garra et al.\cite{regularizedprabhakar}, the authors studied the Laplace transform of the Hilfer-Prabhakar and its regularized version, and applied these results to equations in mathematical physics such as heat and free electron laser equations. Similarly, Panchal et al.\cite{sumudu} applied the Sumudu transform to a non-homogeneous Cauchy problem, and Yudhveer et al.\cite{Singh and Yudhveer} used the Elzaki transform in combination with the Hilfer-Prabhakar fractional derivative and its regularized version to solve a free electron laser type integro-differential equation.
Belgacem et al.\cite{shehu of hilfer-prabhakar} applied the Shehu transform on Prabhakar and Hilfer-Prabhakar derivatives and used it to find solutions for some fractional differential equations. Similarly, the Sawi transform has a deeper connection with Laplace, Elzaki, Sumudu, Kamal, and Monand transforms. In 2019, Mahgoub et al. \cite{Mahgoub2019} introduced a new integral transform called the Sawi transform, with the main purpose of studying its applicability for solving linear differential equations. Nowadays, the Sawi transform is widely used by researchers in science and engineering to solve various problems for integral and differential equations \cite{Higazy2021,AggarwalSudhanshu2019,Patil2021,Higazy2020,Bhuvaneswari2021,Sahoo2022  }.

In this paper, the Sawi transform is used to solve fractional differential equations by applying it to various forms of fractional derivatives, including the Prabhakar integral, Prabhakar derivatives, Hilfer-Prabhakar derivative, and their regularized versions. The findings of these research papers are then used to analyze Cauchy-type fractional differential equations that involve the Hilfer-Prabhakar fractional derivative of fractional order, in terms of the Mittag-Leffler type function.

\section{Definitions and preliminaries}
\begin{defn} \cite{Mahgoub2019}
The Sawi transform, represented by $\mathcal {T}(s)$, is applied to the function $\psi(t)$:

\begin{equation} \label{Sawi tranform}
\begin{split}
Sa[\psi(t),s] = \mathcal {T}(s) =\frac{1}{ s^2} \int_0^\infty \psi (t)& exp \left(\frac{-t}{s}\right)dt\\&
=\frac{1}{s}\int_{0}^{\infty} exp(-t)\psi(st)dt, s\in(\lambda_1,\lambda_2),
\end{split}
\end{equation}
he Sawi transform is applied to a set of functions
\begin{equation*}
\mathcal{A}= \left \{\psi(t)~|~\exists~ M,\lambda_1,\lambda_2 >0, k>0, |\psi(t)| \leq M~ e^{\left(\frac{t}{\lambda_j}\right)}, if~ t\in (-1)^j\times[0,\infty) \right \},
\end{equation*}
The integral transform (\ref{Sawi tranform}) is defined for all values of $\psi(t)$ that are greater than a certain constant value "k".
 
It can be easily established that the Sawi integral transform is a linear operator and has many characteristics that are similar to other integral transforms.
\end{defn}

\begin{prop} \cite{Aggarwal2020}
The Sawi transform to the convolution of two functions, $\psi(t)$ and $\phi(t)$, represented by $\mathcal {F}(s)$ and $\mathcal{G}(s)$ respectively, can be found by applying the Sawi transform to their convolution

\begin{equation}\label{convolution}
Sa[\psi(t)*\phi(t)),s]=s^2 \mathcal{F}(s) \mathcal{G}(s),
\end{equation}
 
or equivalently,
\begin{equation} \label{9}
T^{-1}\left[s^2 \mathcal{F}(s) \mathcal{G}(s),t\right]=  \psi(t)*\phi(t),
\end{equation}

where
\begin{equation}
\psi(t){*}\phi(t) = \int_0^{\infty} \psi(\tau)\phi(t-\tau)d\tau,  
\end{equation} 
\end{prop}

\begin{defn}
\cite{10.5555/1137742}
The Reimann-Liouville integral operator, which is of order $\alpha$ and greater than 0, is applied to a function $\psi(t)$
\begin{equation}\label{RL Iintegral}
_0 \mathcal{I} _t ^\alpha \psi(t) =  \frac{1}{\Gamma(\alpha)} \int_0^t (t-\tau)^{\alpha-1} \psi(\tau) d\tau, \alpha \in \mathbb{C}~and~ t>0.
\end{equation}
\end{defn}

\begin{defn} \cite{10.5555/1137742}
The Reimann-Liouville fractional derivative, which is of order $\alpha$ and greater than 0, is applied to a function $\psi(t)$
\begin{equation} \label{RL Derivative}
_0 \mathcal{D}_t^\alpha \psi(t)=\frac{1}{\Gamma (n-\alpha)} \frac{d^n}{d t^n} \int_0 ^t (t-\tau)^{n-\alpha-1} \psi(\tau) d\tau,~~ n -1<\alpha<n, ~n\in \mathbb{N}. 
\end{equation}

\begin{defn} \cite{10.5555/1137742}
 The Caputo fractional derivative, which is of order $\alpha$ and greater than 0, is applied to a function $\psi(t)$
\end{defn}
\begin{equation} \label{Caputo FD}
_0^C \mathcal{D}_t^\alpha \psi(t)=\frac{1}{\Gamma (n-\alpha)} \int_0 ^t (t-\tau)^{n-\alpha-1} \psi^{(n)}(\tau) d\tau,~~n -1<\alpha<n, ~~n\in \mathbb{N}
\end{equation}
\end{defn}

\begin{defn}\cite{doi:10.1142/9789812817747_0002}
The Hilfer fractional derivative, which is of order $\alpha$ and $\rho$ where $0<\alpha \leq 1$, and $0\leq \rho \leq 1$, is applied to a function $\psi(t)$
\begin{equation} \label{Hilfer FD}
_0 \mathcal{D}_t^{\alpha,\rho} \psi(t)=\left(_0 \mathcal{I}_t^{\rho(1-\alpha)} \frac{d}{dt} ( _0 \mathcal{I}_t^{(1-\alpha)(1-\rho)} \psi(t))   \right) 
\end{equation}
\end{defn}

\begin{defn}\cite{Millerkweyl1975,Saxenaweyl2006}The Weyl fractional differential operator, which is of order $\alpha$ and greater than 0, is applied to a function $\psi(t)$
\begin{equation}\label{weyl FD}
_{-\infty}\mathcal{D}_t^\alpha \psi(t) = \frac{1}{\Gamma(n-\alpha)}\frac{d^n}{dt^n}\int_{-\infty}^{t}(t-\tau)^{n-\alpha-1}\psi(\tau)d\tau,~~n-1<\alpha<n,~~n\in \mathbb{N} 
\end{equation}
The modified Fourier transform of the operator (\ref{weyl FD}), as defined by Metzler and Klafter in their publication \cite{MetzlerandKlafter} is provided,
\begin{equation}\label{fourier of weyl}
F\left\{_{-\infty}\mathcal{D}_t^\alpha \psi(x) \right\} =-k^\alpha \psi^*(k),
\end{equation}
\end{defn}
where $\psi^*(k)$ is the Fourier transform of $\psi(x)$

\begin{defn}\cite{DebnathBhatta2007}
Let $\psi(x)$ be a piecewise continuous function defined on $(-\infty, \infty)$ in each partial interval and absolutely integrable in $(-\infty,\infty)$ the Fourier transform is defined by the integral equation provided,
\begin{equation}\label{fourier tranform}
F[\psi(x),k]=\psi^*(k)=\int_{-\infty}^{\infty}\psi(x)exp(ikx)dx,
\end{equation}
and inverse of Fourier is
\begin{equation}\label{inverse of fourier}
\psi^{-1}[\psi^*(k)]=\frac{1}{2\pi}\int_{-\infty}^\infty \psi(k)exp(-ikx)dk,
\end{equation}
\end{defn}

\begin{defn}\cite{prabhakar1971singular}
 A one-parameter Mittag-Leffler function is provided,
\begin{equation}
E_{\alpha}(z) =\sum_{k=0}^\infty \frac{z^k}{\Gamma(\alpha k +1)},~~z, \alpha \in \mathbb{C}, ~~Re(\alpha)>0.
\end{equation}
A two-parameter Mittag-Leffler function is provided,
\begin{equation}
E_{\alpha,\rho}(z) =\sum_{k=0}^\infty \frac{(z^k)}{\Gamma(\alpha k +\rho)},~~z, \alpha, \rho \in \mathbb{C}, ~~Re(\alpha)>0.
\end{equation}
\end{defn}

\begin{defn}\cite{prabhakar1971singular}
A three-parameter Mittag-Leffler function, also known as the Prabhakar function is provided,
\begin{equation}\label{three pera M-L fun}
E_{\alpha,\rho}^{\gamma}(z)=\sum_{k=0}^{\infty}\frac{(\gamma)_{k}}{\Gamma(\alpha k+\rho)}\frac{(z)^k}{k!},~~~ z, \alpha,\rho,\gamma \in \mathbb{C},\alpha>0,
\end{equation}

for applications purpose we will use further generalization of (\ref{three pera M-L fun}) which is obtained
\begin{equation}\label{mittage in generalization form}
e_{\alpha,\rho,\omega}^{\gamma}=t^{\rho-1}E_{\alpha,\rho}^{\gamma}(\omega t^\alpha)
\end{equation}
the function is provided in terms to the parameter $\omega \in \mathbb{C}$, and the independent variable t, which is a real variable greater than 0.
\end{defn}

\begin{defn}\cite{prabhakar1971singular}
The Prabhakar fractional integral for the function $\psi$ in $L^1 [0,1], 0<b< \infty$, and $t>0$ can be represented in the given format
\begin{equation}\label{prabhakar integral}
\begin{split}
\mathcal{I}_{\alpha,\rho,\omega,0^+}^{\gamma}\psi(t)&=\int_{0}^{t}(t-\tau)^{\rho-1}E_{\alpha,\rho}^{\gamma}(\omega (t-\tau)^\alpha) \psi(\tau)d\tau\\&
=(\psi *e_{\alpha,\rho,\omega}^{\gamma})(t)
\end{split}
\end{equation}
the parameters $\alpha,\rho,\gamma, \omega$ in the given representation of the Prabhakar fractional integral are complex numbers and $\alpha, \rho$ are greater than 0.
\end{defn}

\begin{defn}\cite{prabhakar1971singular}
The Prabhakar fractional derivative for the function $\psi$ in  $L^1 [0,1], 0<b< \infty, t>0$ and $\psi*e_{\alpha,\rho,\omega}^\gamma \in W^{k,1}[0,b], k= \lceil \rho \rceil$,  can be expressed in the given formatat
\begin{equation}\label{prabhakar derivative}
\mathcal{D}_{\alpha,\rho,\omega,0^+}^{\gamma} \psi(t)=
\frac{d^k}{dt^k}\mathcal{I}_{\alpha,k-\rho,\omega,0^+}^{-\gamma} \psi(t),
 \end{equation}
the parameters $\alpha, \rho, \gamma,\omega$ in the given representation of the Prabhakar fractional derivative are complex numbers and the real parts of $\alpha$ and $\rho$ are greater than 0.

The Reimann Liouville Fractional Derivative in (\ref{RL Derivative}) can be expressed in the given format
\begin{equation}
\mathcal{D}_{\alpha,\rho,\omega,0^+}^{\gamma} \psi(t)=\mathcal{D}_{0^+}^{\rho+\theta} \mathcal{I}_{\alpha,\theta,\omega,0^+}^{-\gamma}\psi(t),
\end{equation}
\end{defn}

\begin{defn}\cite{regularizedprabhakar}
The regularized Prabhakar fractional derivative for the function $\psi$ in $AC[0,b], 0<b<\infty, t>0$, and $k= \lceil \rho \rceil$,  can be expressed in the given format
\begin{equation}\label{regularized prabhakar}
^C\mathcal{D}_{\alpha,\rho,\omega,0^+}^{\gamma}\psi(t) =\mathcal{I}_{\alpha,k-\rho,\omega,0^+}^{-\gamma} \frac{d^k}{dt^k} \psi(t)
\end{equation}
where $\alpha$, $\rho$, $\gamma$, $\omega$ are complex numbers, and the real parts of $\alpha$ and $\rho$ are greater than 0.
\end{defn}

\begin{defn}\cite{regularizedprabhakar,k-Hilfer-Prabhakar}
Let  $\psi \in L^1 [a,b],\rho \in (0,1), \nu \in [0,1], 0<b<t\leq \infty, ~~ \psi *e_{\alpha,(1-\nu)(1-\rho),\omega}^{-\gamma(1-\nu)}(.) \in AC^1[a,b]$. The \textbf{Hilfer-Prabhakar fractional derivative} is obtained
\begin{equation}\label{hilfer-prabhar derivative}
\mathcal{D}_{\alpha,\omega,0^+}^{\gamma,\rho,\nu} \psi(t)  =  \left(\mathcal{I}_{\alpha,\nu(1-\rho),\omega,0^+}^{-\gamma\nu} \frac{d}{dt} (\mathcal{I}_{\alpha,(1-\nu)(1-\rho),\omega,0^+}^{-\gamma(1-\nu)}\psi )\right)(t)
\end{equation}
\end{defn}

\begin{defn}\cite{k-Hilfer-Prabhakar} 
\textbf{The regularized Hilfer-Prabhakar fractional derivative} to the function $\psi(t)$ in $L^1 [a,b],\rho \in (0,1), \nu \in [0,1], 0<b\leq \infty, t>0$ is obtained  
\begin{equation}\label{regularized hilfer-prabhar d}
\begin{split}
^C\mathcal{D}_{\alpha,\omega,0^+}^{\gamma,\rho,\nu} \psi(t)&
=  \left(\mathcal{I}_{\alpha,\nu(1-\rho),\omega,0^+}^{-\gamma\nu}\mathcal{I}_{\alpha,(1-\nu)(1-\rho),\omega,0^+}^{-\gamma(1-\nu)}\frac{d}{dt}\psi \right)(t)\\&
= \mathcal{I}_{\alpha,1-\rho,\omega,0^+}^{-\gamma}\frac{d}{dt}\psi(t)
\end{split}
\end{equation}
\end{defn}

\begin{thm}  \cite{Mahgoub2019}
If $\mathcal{T} (s)$ represents the Sawi transform to the function $\psi(t)$, then the Sawi transform to the mth derivative of $\psi(t)$, represented by $\mathcal {T}_m$(s) 
\begin{equation}\label{mth Sawi transform} 
\mathcal{T}_m(s)=Sa[\psi^{(m)}(t),s] = s^{-m}\mathcal{T}(s)-  \sum_{k=0}^{m-1} s^{k-m-1}\psi^{(k)}(0), m\geq 0
\end{equation}
\end{thm}

\begin{lem}
Let $0<\alpha<1$ and $\omega \in \mathbb{C} $ such that $Re(\alpha) >0, ~Re(\rho)>0),~Re(\gamma)>0$. The Sawi
transform of Mittage-Leffler type function $t^{\rho-1} E_{\alpha,\rho}^\gamma (\omega t^\alpha) $ by using (\ref{Sawi tranform}), (\ref{three pera M-L fun}) is obtained
\begin{equation}\label{Sawi of 3 pera M-L function }
Sa \left[t^{\rho-1} E_{\alpha,\rho}^\gamma (\omega t^\alpha) ,s \right]
= s^{\rho-2} (1-\omega s^\alpha)^{-\gamma},
\end{equation}
\end{lem}
\begin{proof}
\begin{equation*}
\begin{split}
Sa \left[t^{\rho-1} E_{\alpha,\rho}^\gamma (\omega t^\alpha) ,s \right]&
= \frac{1}{s}\int_{0}^{\infty} e^{-t}(st)^{\rho-1}E_{\alpha,\rho}^{\gamma}(\omega(st)^\alpha)dt,\\&
=s^{\rho-2}\sum_{k=0}^\infty \frac{(\gamma)_k (\omega s^\alpha)^k}{\Gamma(\alpha k+\rho) k!}\int_{0}^\infty e^{-t}t^{\alpha k+\rho-1}dt,\\&
=s^{\rho-2}\sum_{k=0}^{\infty}\frac{(\gamma)_k (\omega s^\alpha)^k}{k!},\\&
=s^{\rho-2}(1-\omega s^\alpha)^{-\gamma}.
\end{split}
\end{equation*}
\end{proof}

\begin{lem}
Given that the Sawi transform to the function $\psi(t)$ representation as $\mathcal{T}(s)$, the Sawi transform to the Prabhakar fractional integral of $\psi(t)$ can be found by using equations (\ref{convolution}), (\ref{Sawi of 3 pera M-L function }), and represented as shown below:
\begin{equation}\label{Sawi of prabhakar integral}
Sa[\mathcal{I}_{\alpha,\rho,\omega,0^+}^{\gamma} \psi(t),s]
=s^\rho (1-\omega s^\alpha)^{-\gamma} \mathcal{T}(s),
\end{equation}
\begin{proof}
\begin{equation*}
\begin{split}
Sa[\mathcal{I}_{\alpha,\rho,\omega,0^+}^{\gamma} \psi(t),s]& 
= Sa \left[\int_0^t (t-\tau)^{\rho-1} E_{\alpha,\rho}^\gamma[\omega (t-\tau)^\alpha]\psi(\tau)d\tau ,s \right] \\&
= Sa \left[ (\psi *e_{\alpha,\rho,\omega}^\gamma)(t),s \right]\\&
= s^2 Sa \left[t^{\rho-1} E_{\alpha,\rho}^\gamma (\omega t^\alpha) ,s \right]\times Sa[\psi(t),s]\\&
=s^2\times s^{\rho-2}(1-\omega s^\alpha)^{-\gamma}\mathcal{T}(s)  \\& 
=s^\rho(1-\omega s^\alpha)^{-\gamma} \mathcal{T}(s),
\end{split}
\end{equation*}
\end{proof}
\end{lem}

\section{Main results}

\begin{thm}\textbf{[Sawi transform of Prabhakar derivative]:}The Sawi transform to the Prabhakar fractional derivative representation in the given format
\begin{equation}\label{Sawi of prabhakar derivative}
Sa[\mathcal{D}_{\alpha,\rho,\omega,0^+}^{\gamma} \psi(t),s ]
=s^{-\rho} (1-\omega s^\alpha)^\gamma \mathcal{T}(s)  
- \sum_{k=0}^{m-1} s^{k-m-1}  \mathcal{D}_{\alpha,k-m+\rho,\omega,0^+}^{\gamma} \psi(t)|_{t=0}
\end{equation}
\end{thm}
\begin{proof} Given that the Sawi transform to the function $\psi(t)$ representation as $\mathcal{T}(s)$, by applying the Sawi transform on the Prabhakar fractional derivative (\ref{prabhakar derivative}) w.r.t.  variable $t$, and using equations (\ref{mth Sawi transform}) and convolution (\ref{convolution}), we obtain the given representation
\begin{equation*}
\begin{split}
Sa&[\mathcal{D}_{\alpha,\rho,\omega,0^+}^{\gamma} \psi(t),s ] \\&
= Sa \left[\frac{d^m}{dt^m}\mathcal{I}_{\alpha,m-\rho,\omega,0^+}^{-\gamma} \psi(t),s \right]\\&
= Sa \left[\frac{d^m}{dt^m} g(t),s \right],~~where~~~ g(t)=\mathcal{I}_{\alpha,m-\rho,\omega,0^+}^{-\gamma} \psi(t) \\&
= s^{-m} Sa[g(t),s]- \sum_{k=0}^{m-1} s^{k-m-1} g^{(k)}(0),~~g^{(k)}(0)= \frac{d^k}{dt^k} \mathcal{I}_{\alpha,m-\rho,\omega,0^+}^{-\gamma} \psi(0) \\&
=s^{-\rho} (1-\omega s^\alpha)^{\gamma} \mathcal{T}(s)- \sum_{k=0}^{m-1} s^{k-m-1}  \left[\mathcal{D}_{\alpha,k-m+\rho,\omega,0^+}^{\gamma} \psi(t)\right]_{t=0}
\end{split}
\end{equation*}
\end{proof}

\begin{thm}\textbf{[Sawi transform of regularised Prabhakar derivative]:} The Sawi transform to the regularized Prabhakar fractional derivative representation in the given format
\begin{equation}\label{Sawi of reg prabhakar}
Sa[^C\mathcal{D}_{\alpha,\rho,\omega,0^+}^{\gamma} \psi(t),s ]
=s^{-\rho} (1-\omega s^\alpha)^{\gamma} \mathcal{T}(s)
- \sum_{k=0}^{m-1} s^{k-\rho-1} (1-\omega s^\alpha)^{\gamma} \psi^{(k)}(0^+)
\end{equation}
\end{thm}
\begin{proof}Given that the Sawi transform to the function $\psi(t)$ representation as $\mathcal{T}(s)$, by applying the Sawi transform on the regularized Prabhakar fractional derivative (\ref{regularized prabhakar}) w.r.t.  variable $t$, and using equations (\ref{Sawi of prabhakar integral}), (\ref{mth Sawi transform}), and convolution (\ref{convolution}) of Sawi transform, we obtain the given representation
\begin{equation*}
\begin{split}
Sa& [^C\mathcal{D}_{\alpha,\rho,\omega,0^+}^{\gamma} \psi(t),s ] \\ &
= Sa\left[\mathcal{I}_{\alpha,m-\rho,\omega,0^+}^{-\gamma} \frac{d^m}{dt^m} \psi(t),s\right]\\&
= Sa\left[\mathcal{I}_{\alpha,m-\rho,\omega,0^+}^{-\gamma} h(t),s\right], ~~~ where~~~ h(t)=\frac{d^m}{dt^m} \psi(t) \\&
= s^{m-\rho} (1-\omega s^\alpha)^\gamma Sa[h(t),s]\\&
= s^{m-\rho} (1-\omega s^\alpha)^\gamma \left[s^{-m} Sa[\psi(t),s] - \sum_{k=0}^{m-1}s^{k-m-1} \psi^{(k)}(0) \right]\\&
=  s^{-\rho}(1-\omega s^\alpha)^{\gamma} \mathcal{T}(s)
- \sum_{k=0}^{m-1} s^{k-\rho-1} (1-\omega s^\alpha)^{\gamma} \psi^{(k)}(0^+)
\end{split}
\end{equation*}
\end{proof}

\begin{thm}\textbf{[Sawi transform of Hilfer-Prabhakar  derivative]:} The Sawi transformation expresses the Hilfer-Prabhakar fractional derivative in a specific form
\begin{equation} \label{Sawi of hilfer-prabhakar}
\begin{split}
Sa[\mathcal{D}_{\alpha,\omega,0^+}^{\gamma,\rho,\nu} \psi(t),s ]= s^{-\rho} (1-\omega  s^\alpha &)^\gamma  \mathcal{T}(s) \\ & 
-s^{\nu(1-\rho)-2} (1-\omega s^\alpha)^{\gamma \nu}~ \mathcal{I}_{\alpha,(1-\nu)(1-\rho),\omega,0^+}^{-\gamma(1-\nu)}\psi(t)|_{t=0^+}
\end{split}
\end{equation}
\end{thm}
\begin{proof}Given that the Sawi transform to the function $\psi(t)$ representation as $\mathcal{T}(s)$, by applying the Sawi transform on the Hilfer-Prabhakar fractional derivative (\ref{hilfer-prabhar derivative}) w.r.t.  variable $t$, and using equations (\ref{Sawi of prabhakar integral}), (\ref{mth Sawi transform}), we obtain the given representation
\begin{equation*}
\begin{split}
Sa&[\mathcal{D}_{\alpha,\omega,0^+}^{\gamma,\rho,\nu} \psi(t),s ]\\&
=Sa\left[\left(\mathcal{I}_{\alpha,\nu(1-\rho),\omega,0^+}^{-\gamma\nu} \frac{d}{dt} (\mathcal{I}_{\alpha,(1-\nu)(1-\rho),\omega,0^+}^{-\gamma(1-\nu)}\psi )\right)(t),s\right]\\&
=Sa\left[\mathcal{I}_{\alpha,\nu(1-\rho),\omega,0^+}^{-\gamma\nu} k(t),s  \right],~~where~~ k(t)=\frac{d}{dt}\mathcal{I}_{\alpha,(1-\nu)(1-\rho),\omega,0^+}^{-\gamma(1-\nu)}\psi(t)\\&
=s^{\nu(1-\rho)} (1-\omega s^\alpha)^{\gamma\nu} Sa[k(t),s]\\&
=s^{\nu(1-\rho)} (1-\omega s^\alpha)^{\gamma\nu} \left[s^{-1} Sa[\mathcal{I}_{\alpha,(1-\nu)(1-\rho),\omega,0^+}^{-\gamma (1-\nu)}\psi(t),s ] -s^{-2} \mathcal{I}_{\alpha,(1-\nu)(1-\rho),\omega,0^+}^{-\gamma (1-\nu)}\psi(0^+)  \right]\\&
= s^{-\rho} (1-\omega s^\alpha)^\gamma \mathcal{T}(s)
-s^{\nu(1-\rho)-2} (1-\omega s^\alpha)^{\gamma \nu}~ \mathcal{I}_{\alpha,(1-\nu)(1-\rho),\omega,0^+}^{-\gamma(1-\nu)}\psi(t)|_{t=0^+}
\end{split}
\end{equation*}
\end{proof}

\begin{thm}\textbf{[Sawi transform of regularized Hilfer-Prabhakar derivative]:} The Sawi transform to the regularized Hilfer-Prabhakar fractional derivative representation in the given format
\begin{equation}\label{Sawi of reg hilfer-prabhakar}
Sa[^C\mathcal{D}_{\alpha,\omega,0^+}^{\gamma,\rho,\nu} \psi(t),s ]=  s^{-\rho} (1-\omega s^\alpha)^\gamma \mathcal{T}(s)
- s^{-\rho-1} (1-\omega s^\alpha)^{\gamma} f(0^+)
\end{equation}
\end{thm}
\begin{proof}Given that the Sawi transform to the function $\psi(t)$ representation as $\mathcal{T}(s)$, by applying the Sawi transform on the regularized Hilfer-Prabhakar fractional derivative (\ref{regularized hilfer-prabhar d}) w.r.t.  variable $t$, and using equations (\ref{Sawi of prabhakar integral}), (\ref{mth Sawi transform}), and convolution (\ref{convolution} ) of Sawi transform, we obtain the given representation
\begin{equation*}
\begin{split}\label{21}
Sa[^C\mathcal{D}_{\alpha,\omega,0^+}^{\gamma,\rho,\nu} \psi(t),s ]&
=Sa\left[\mathcal{I}_{\alpha,1-\rho,\omega,0^+}^{-\gamma} \frac{d}{dt}\psi(t) ,s\right]\\&
=Sa\left[\mathcal{I}_{\alpha,1-\rho,\omega,0^+}^{-\gamma} z(t) ,s\right], ~~z(t)=\frac{d}{dt}\psi(t)\\&
=s^{1-\rho}(1-\omega s^\alpha)^\gamma Sa[z(t),s]\\&
=s^{1-\rho}(1-\omega s^\alpha)^\gamma [s^{-1} Sa[\psi(t),s]-s^{-2}f(0^+)]\\&
=s^{-\rho}(1-\omega s^\alpha)^\gamma \mathcal{T}(s)-s^{-\rho-1}(1-\omega s^\alpha)^\gamma f(0^+)\\&
\end{split}
\end{equation*}
\end{proof}

\section{Applications}
This section presents the application of Sawi transform on Hilfer-Prabhakar and regularized Hilfer-Prabhakar fractional derivatives to solve Cauchy-type fractional differential equations,
\begin{thm}
This theorem presents the solution for the generalized Cauchy-type problem for the fractional advection dispersion equation
\begin{equation}\label{hilfer-prabhakar problem with weyl}
\mathcal{D}_{\alpha,\omega,0^+}^{\gamma, \rho,\nu}\psi(x,t)=-p \mathcal{D}_{x} \psi(x,t) + \vartheta~ \Delta^\frac{\lambda}{2} \psi(x,t) 
\end{equation}
subjects to below constraints
\begin{equation}
\mathcal{I}_{\alpha,(1-\nu)(1-\rho),\omega,0^+}^{-\gamma(1-\nu)}  \psi(x,0^+)=g(x), ~~\omega,~\gamma, ~x\in \mathbb{R},~\alpha>0,
\end{equation}
\begin{equation}
\lim _{x\to\infty}\psi(x,t)=0,~~t\geq 0,~~ 
\end{equation}
is obtained
\begin{equation}
\psi(x,t)=\frac{1}{2\pi}\int_{-\infty}^{\infty}\sum_{n=0}^{\infty}t^{\nu(1-\rho)+n\rho+\rho-1}e^{(-ikx)} g(k)(i p k-\vartheta|k|^\lambda)^n E_{\alpha,\nu(1-\rho)+\rho(n+1)}^{\gamma(1+n)-\gamma \nu}(\omega t^\alpha)dk
\end{equation}
In the solution, $\Delta^\frac{\lambda}{2}$ represents the fractional generalized Laplace operator of order $\lambda$, where $\lambda$ is in the range of (0,2), $\rho$ is in the range of (0,1), $\nu$ is in the range of [0,1], $x$ is in $\mathbb{R}$, $t$ is in $\mathbb{R^+}$, and $\gamma$ is greater than 0. The Fourier transform of $ \Delta^\frac{\lambda}{2} $ is $-|k|^\lambda$, as discussed in the publication \cite{agarwal2016}
\end{thm}
\begin{proof}The function $\overline \psi(x,s)$ represents the Sawi transform of $\psi(x,t)$ w.r.t.  variable $t$, and $\psi^*(k,t)$ represents the Fourier transform of $\psi(x,t)$ w.r.t.  variable $x$. By applying the Fourier and Sawi transform on equation (\ref{hilfer-prabhakar problem with weyl}), using equations (\ref{Sawi of hilfer-prabhakar}), (\ref{fourier of weyl}), the transformed solution is obtained,
first we will apply Sawi and Fourier transform on left hand side
\begin{equation}
\begin{split}
LHS& = s^{-\rho}(1-\omega s^\alpha)^\gamma \overline \psi^*(k,s)-s^{\nu(1-\rho)-2}(1-\omega s^\alpha)^{\gamma \nu}\mathcal{I}_{\alpha,(1-\nu)(1-\rho),\omega,0^+}^{-\gamma(1-\nu)}\psi^*(k,0)\\&
=s^{-\rho}(1-\omega s^\alpha)^\gamma \overline \psi^*(k,s)-s^{\nu(1-\rho)-2}(1-\omega s^\alpha)^{\gamma \nu}g^*(k)
\end{split}
\end{equation}
now we will apply the Sawi and Fourier transform on right hand side
\begin{equation}
RHS=  i p k \overline\psi^*(k,s)-\vartheta |k|^\lambda \overline\psi^*(k,s)
\end{equation}
therefore, we will get
\begin{equation*}
\begin{split}
s^{-\rho}&(1-\omega s^\alpha)^\gamma \overline \psi^*(k,s)
-s^{\nu(1-\rho)-2}(1-\omega s^\alpha)^{\gamma \nu}g^*(k)
=i p k \overline\psi^*(k,s)-\vartheta |k|^\lambda \overline\psi^*(k,s)\\&
\overline \psi^*(k,s)[s^{-\rho}(1-\omega s^\alpha)^\gamma +\vartheta|k|^\lambda - i p k]=s^{\nu(1-\rho)-2}(1-\omega s^\alpha)^{\gamma \nu}g^*(k)\\&
\overline \psi^*(k,s)=\frac{s^{\nu(1-\rho)-2} (1-\omega s^\alpha)^{\gamma \nu}g^*(k) } 
{ s^{-\rho}(1-\omega s^\alpha)^\gamma   \left[1+\frac{\vartheta |k|^\lambda -i p k}{ s^{-\rho}(1-\omega s^\alpha)^\gamma}\right] }\\&
\overline \psi^*(k,s)=s^{\nu(1-\rho)+\rho-2}(1-\omega s^\alpha)^{\gamma \nu-\gamma}g^*(k)\sum_{n=0}^{\infty}\left[\frac{-\vartheta |k|^\lambda + i p k}{ s^{-\rho}(1-\omega s^\alpha)^\gamma}\right]^n\\&
\overline \psi^*(k,s)=\sum_{n=0}^{\infty}( i p k-\vartheta |k|^\lambda )^n s^{\nu(1-\rho)+\rho +\rho n-2}(1-\omega s^\alpha)^{\gamma \nu-\gamma n-\gamma}g^*(k)
\end{split}
\end{equation*}
for $\frac{\vartheta |k|^\lambda -ip k}{ s^{-\rho}(1-\omega s^\alpha)^\gamma} <1$, applying the inverse of Fourier $\psi^{-1}[\psi^*(k)] =\frac{1}{2\pi}\int_{-\infty}^{\infty}e^{-i k x}\psi(k)dk $ and Sawi $Sa^{-1}[s^{\rho -2}(1-\omega s^\alpha)^{-\gamma}]= t^{\rho-1}E_{\alpha,\rho}^{\gamma}(\omega t^\alpha) $ transforms and we will get the desired result
\begin{equation*}
\psi(x,t)=\frac{1}{2\pi}\int_{-\infty}^{\infty}\sum_{n=0}^{\infty}t^{\nu(1-\rho)+n\rho+\rho-1}e^{(-ikx)} g(k)(ipk-\vartheta|k|^\lambda)^n E_{\alpha,\nu(1-\rho)+\rho(n+1)}^{\gamma(1+n)-\gamma \nu}(\omega t^\alpha)dk
\end{equation*}
\end{proof}


\begin{thm}
This theorem provides a solution for the generalized Cauchy-type problem for the fractional advection dispersion equation.
\begin{equation}\label{regu hilfer-prabhakar problem with p and rho}
^C\mathcal{D}_{\alpha,\omega,0^+}^{\gamma, \rho,\nu}\psi(x,t)=-p \mathcal{D}_{x} \psi(x,t) + \vartheta~ \Delta^\frac{\lambda}{2} \psi(x,t) 
\end{equation}
subjects to below constraints
\begin{equation}
 \psi(x,0^+)=g(x), ~x\in \mathbb{R},
 \end{equation}
 \begin{equation}
\lim _{|x|\to\infty}\psi(x,t)=0,~~t\geq 0,~~ 
\end{equation}
is obtained
\begin{equation}
\psi(x,t)=\frac{1}{2\pi}\int_{-\infty}^{\infty}\sum_{n=0}^{\infty}t^{\rho n}e^{(-ikx)} g(k)(i p k-\vartheta|k|^\lambda) E_{\alpha,\rho(n+1)}^{\gamma n}(\omega t^\alpha)dk
\end{equation}
In the solution, $\Delta^\frac{\lambda}{2}$ represents the fractional generalized Laplace operator of order $\lambda$, where $\lambda$ is in the range of (0,2), $\rho$ is in the range of (0,1), $\nu$ is in the range of [0,1], $x$ is in $\mathbb{R}$, $t$ is in $\mathbb{R^+}$, and $\gamma$ is greater than 0, and Fourier transform of $ \Delta^\frac{\lambda}{2} $ is $-|k|^\lambda$ discussed in \cite{agarwal2016}
\end{thm}
\begin{proof} The function $\overline \psi(x,s)$ represents the Sawi transform of $\psi(x,t)$ w.r.t.  variable $t$, and $\psi^*(k,t)$ represents the Fourier transform of $\psi(x,t)$ w.r.t.  variable $x$. By applying the Fourier and Sawi transform on equation (\ref{regu hilfer-prabhakar problem with p and rho}), using equations (\ref{Sawi of reg hilfer-prabhakar}), (\ref{fourier of weyl}), the transformed solution is obtained
\begin{equation*}
\begin{split}
s^{-\rho}&(1-\omega s^\alpha)^\gamma \overline \psi^*(k,s)-s^{-\rho-1}(1-\omega s^\alpha)^\gamma \psi ^*(k,0)=ipk \overline \psi^*(k,s)-\vartheta|k|^\lambda \overline \psi^*(k,s)\\&
\overline \psi^*(k,s)[s^{-\rho}(1-\omega s^\alpha)^\gamma+\vartheta|k|^\lambda-ipk]=s^{-\rho-1}(1-\omega s^\alpha)^\gamma g^*(k)\\&
\overline \psi^*(k,s)=\frac{s^{-\rho-1}(1-\omega s^\alpha)^\gamma g^*(k)}{s^{-\rho}(1-\omega s^\alpha)^\gamma \left[1+\frac{\vartheta|k|^\lambda-ipk}{s^{-\rho}(1-\omega s^\alpha)^\gamma} \right] }\\&
\overline \psi^*(k,s)=s^{-1}g^*(k)\sum_{n=0}^\infty \left[\frac{ipk-\vartheta|k|^\lambda}{s^{-\rho}(1-\omega s^\alpha)^\gamma} \right]^n\\& 
\overline \psi^*(k,s)=g^*(k)\sum_{n=0}^\infty(ipk-\vartheta|k|^\lambda)^n s^{\rho n-1}(1-\omega s^\alpha)^{-\gamma n}
\end{split}
\end{equation*}
for $\left[\frac{\vartheta|k|^\lambda-ipk}{s^{-\rho}(1-\omega s^\alpha)^\gamma} \right] <1 $, taking the inverse of Sawi and Fourier on both sides of above equation, we will get the final result
\begin{equation*}
\psi(x,t)=\frac{1}{2\pi}\int_{-\infty}^{\infty}\sum_{n=0}^{\infty}t^{\rho n}e^{(-ikx)} g(k)(i p k-\vartheta|k|^\lambda) E_{\alpha,\rho(n+1)}^{\gamma n}(\omega t^\alpha)dk
\end{equation*}

\end{proof}

This section examines the Fractional Free Electron Laser (FEL) equation that includes the Hilfer-Prabhakar fractional derivative.
\begin{thm}
This section presents the solution for the generalized Cauchy-type problem for the fractional heat equation
\begin{equation} \label{reg hilfer-prabhakar problem with N}
^C\mathcal{D} _{\alpha,\omega,0^+}^{\gamma,\rho,\nu} \psi(x,t)
= N \frac{{\partial}^2}{{\partial} x^2} \psi(x,t)
\end{equation}
subject to the initial condition
\begin{equation}\label{f(x,0)=g(x)}
\psi(x,0)=g(x)
\end{equation}
$$\lim_{x\to\infty}\psi(x,t)=0$$
\end{thm}
$with~~\rho\in(0,1),~\nu[0,1];~\omega, x \in \mathbb{R};~ N, \alpha>0,\gamma\geq 0,
~~is~~ given~~ by$
\begin{equation}
\begin{split}
\psi(x,t)
= \frac{1}{2\pi}\int_{-\infty}^{\infty}\sum_{n=0}^{\infty}e^{-ikx} E_{\alpha,\rho n +1}^{\gamma n}(\omega t^\alpha)(-Nk^2)^n t^{\rho n} g(k)dk
\end{split}
\end{equation}
\begin{proof} The function $\overline \psi(x,s)$ represents the Sawi transform of $\psi(x,t)$ w.r.t.  variable $t$, and $\psi^*(k,t)$ represents the Fourier transform of $\psi(x,t)$ w.r.t.  variable $x$. By applying the Fourier and Sawi transform on equation (\ref{reg hilfer-prabhakar problem with N}), using equations (\ref{Sawi of reg hilfer-prabhakar}), (\ref{f(x,0)=g(x)}), (\ref{inverse of fourier}), the transformed solution is obtained
\begin{equation*}
\begin{split}
s^{-\rho}&(1-\omega s^\alpha)^\gamma \overline \psi^*(k,s) -s^{-\rho-1}(1-\omega s^\alpha)^\gamma \psi^*(k,0)=-N k^2 \overline \psi^*(k,s)\\&
\overline \psi^*(k,s)[s^{-\rho}(1-\omega s^\alpha)^\gamma+N k^2 ]= s^{-\rho-1} (1-\omega s^\alpha)^\gamma g^*(k),\\&
\overline \psi^*(k,s)= \frac{ s^{-\rho-1} (1-\omega s^\alpha)^\gamma g^*(k)}
{s^{-\rho}(1-\omega s^\alpha)^\gamma\left[1+\frac{Nk^2}{s^{-\rho}(1-\omega s^\alpha)^\gamma  }    \right]  },\\&
\overline \psi^*(k,s)=s^{-1} g^*(k) \sum_{n=0}^{\infty}\left[\frac{-Nk^2}{s^{-\rho}(1-\omega s^\alpha)^\gamma  }    \right]^n,\\&
\overline \psi^*(k,s)=s^{-1} g^*(k) \sum_{n=0}^{\infty} (-Nk^2)^n s^{\rho n} (1-\omega s^\alpha)^{-\gamma n}g^*(k),\\&
\overline \psi^*(k,s)= \sum_{n=0}^{\infty}(-N k^2)^n s^{\rho n-1}(1-\omega s^\alpha)^{-\gamma n} g^*(k),
\end{split}
\end{equation*}
for $\frac{Nk^2}{s^{-\rho}(1-\omega s^\alpha)^\gamma }<1 $, applying the inverse of Sawi and Fourier transforms on both sides of the above equation and we will get the final result
\begin{equation*}
\psi(x,t)
= \frac{1}{2\pi}\int_{-\infty}^{\infty}\sum_{n=0}^{\infty}e^{-ikx} E_{\alpha,\rho n +1}^{\gamma n}(\omega t^\alpha)(-Nk^2)^n t^{\rho n} g(k)dk
\end{equation*}
\end{proof}

\begin{thm}
This theorem presents the solution for the generalized Cauchy-type problem for the fractional heat equation,
\begin{equation} \label{hilfer-prabhakar prob with Msec}
\mathcal{D} _{\alpha,\omega,0^+}^{\gamma,\rho,\nu} \psi(x,t)
= M \frac{\partial^2}{\partial x^2} \psi(x,t),
\end{equation}
\begin{equation}\label{condition of g(x)}
\mathcal{I}_{\alpha,(1-\nu)(1-\rho),\omega,0^+}^{-\gamma(1-\nu)}\psi(x,t)|_{t=0}=g(x),
\end{equation}
$$\lim_{x\to\infty}\psi(x,t)=0,$$
$with~~\rho\in(0,1),~\nu[0,1];~\omega, x \in \mathbb{R};~ M, \alpha>0,\gamma\geq 0,
~~is~~ given~~ by$
\begin{equation}
\begin{split}
\psi(x,t)
= \frac{1}{2\pi}\int_{-\infty}^{\infty}\sum_{n=0}^{\infty}e^{-ikx}g(k)(-M k^2)^nt^{\rho(n+1)-\nu(\rho-1)-1} E_{\alpha,\rho(n+1)+\nu(1-\rho)}^{\gamma(n+1-\nu)} dk
\end{split}
\end{equation}
\end{thm}
\begin{proof} The function $\overline \psi(x,s)$ represents the Sawi transform of $\psi(x,t)$ w.r.t.  variable $t$, and $\psi^*(k,t)$ represents the Fourier transform of $\psi(x,t)$ w.r.t.  variable $x$. By applying the Fourier and Sawi transform on equation (\ref{hilfer-prabhakar prob with Msec}), using equations (\ref{Sawi of hilfer-prabhakar}), (\ref{condition of g(x)}), (\ref{inverse of fourier}), the transformed solution is obtained,
\begin{equation*}
\begin{split}
s^{-\rho}(1& -\omega s^\alpha )^\gamma \overline \psi^*(k,s)-s^{\nu(1-\rho)-2}(1-\omega s^\alpha)^{\gamma\nu} \mathcal{I}_{\alpha,(1-\nu)(1-\rho),\omega,0^+}^{-\gamma(1-\nu)}f(x,0)
=-Mk^2 \overline \psi^*(k,s)\\&
s^{-\rho}(1-\omega s^\alpha)^\gamma \overline \psi^*(k,s)-s^{\nu(1-\rho)-2}(1-\omega s^\alpha)^{\gamma\nu} g^*(k)
=-Mk^2 \overline \psi^*(k,s)\\&
\overline \psi^*(k,s)[s^{-\rho}(1-\omega s^\alpha)^\gamma+Mk^2]=s^{\nu(1-\rho)-2}(1-\omega s^\alpha)^{\gamma\nu}g^*(k)\\&
\overline \psi^*(k,s)=\frac{s^{\nu(1-\rho)-2}(1-\omega s^\alpha)^{\gamma\nu}g^*(k)}
{s^{-\rho}(1-\omega s^\alpha)^\gamma+Mk^2  }\\&
\overline \psi^*(k,s)=\frac{s^{\nu(1-\rho)-2}(1-\omega s^\alpha)^{\gamma\nu}g^*(k)}
{s^{-\rho}(1-\omega s^\alpha)^\gamma\left[1+\frac{Mk^2}{s^{-\rho}(1-\omega s^\alpha)^\gamma  }   \right]  }\\&
\overline \psi^*(k,s)=s^{\nu(1-\rho)-2}(1-\omega s^\alpha)^{\gamma\nu-\gamma}g^*(k)\sum_{n=0}^{\infty}(-Mk^2)^n s^{\rho n}(1-\omega s^\alpha)^{-\gamma n}\\&
\overline \psi^*(k,s)=g^*(k)\sum_{n=0}^{\infty}(-Mk^2)^n s^{\rho n +\nu(1-\rho)+\rho-2}(1-\omega s^\alpha)^{\gamma \nu-\gamma n-\gamma}
\end{split}
\end{equation*}
for $ \left(\frac{Mk^2}{s^{-\rho}(1-\omega s^\alpha)^\gamma  } \right)<1 $, taking the inverse of Fourier and Sawi transforms, we will get the desired result  
\begin{equation*}
\begin{split}
\psi(x,t)
= \frac{1}{2\pi}\int_{-\infty}^{\infty}\sum_{n=0}^{\infty}e^{-ikx}g(k)(-M k^2)^nt^{\rho(n+1)-\nu(\rho-1)-1} E_{\alpha,\rho(n+1)+\nu(1-\rho)}^{\gamma(n+1-\nu)} dk
\end{split}
\end{equation*}
\end{proof}

\begin{thm} The solution of Cauchy type fractional differential equation
\begin{equation}\label{reg hilfer-prabhakar problem with lambda(1-x)}
^C \mathcal{D}_{\alpha,-\omega,0^+}^{\gamma,\rho,\nu}\psi(x,t)=-\lambda(1-x)\psi(x,t), ~~for~~ |x|\leq 1
\end{equation}
\begin{equation}\label{f(x,0)=1}
f(x,0)=1
\end{equation}
with~~~~~ $ t>0,~\lambda>0,~\gamma\geq 0,~ 0<\alpha\leq 1,0<\rho\leq 1,~~is $
\begin{equation}
\psi(x,t)=\sum_{n=0}^{\infty}(-\lambda)^n(1-x)^n t^{\rho n}E_{\alpha,\rho n+1}^{\gamma n}(-\omega t^\alpha)
\end{equation}
\end{thm}
\begin{proof} The function $\overline \psi(x,s)$ represents the Sawi transform of $\psi(x,t)$ w.r.t.  variable $t$. By applying the Sawi transform on equation (\ref{reg hilfer-prabhakar problem with lambda(1-x)}), using (\ref{Sawi of reg hilfer-prabhakar}), (\ref{f(x,0)=1}), the transformed solution is obtained
\begin{equation*}
\begin{split}
s^{-\rho}&(1+\omega s^\alpha)^\gamma ~\overline \psi(k,s)-s^{-\rho-1}(1+\omega s^\alpha)^\gamma~ \overline \psi(x,0)=-\lambda(1-x)\overline \psi(x,s)\\&
\overline \psi(k,s)[s^{-\rho}(1+\omega s^\alpha)^\gamma +\lambda(1-x)]=s^{-\rho-1}(1+\omega s^\alpha)^\gamma\\&
\overline \psi(k,s)=\frac{s^{-\rho-1}(1+\omega s^\alpha)^\gamma }
{s^{-\rho}(1+\omega s^\alpha)^\gamma\left[1+\frac{\lambda(1-x)}{s^{-\rho}(1+\omega s^\alpha)^\gamma  }  \right]  }\\&
\overline \psi(k,s)=\frac{s^{-\rho-1}(1+\omega s^\alpha)^\gamma }
{s^{-\rho}(1+\omega s^\alpha)^\gamma  }\left[1+\frac{\lambda(1-x)}{s^{-\rho}(1+\omega s^\alpha)^\gamma  }  \right]^{-1}\\&
\overline \psi(k,s)=s^{-1} \sum_{n=0}^{\infty} \left[\frac{-\lambda(1-x)}{s^{-\rho}(1+\omega s^\alpha)^\gamma  }  \right]^n\\&
\overline \psi(k,s)=\sum_{n=0}^{\infty}(-\lambda)^n(1-x)^n s^{\rho n-1}(1+\omega s^\alpha)^{-\gamma n} 
\end{split}
\end{equation*}
for $\frac{\lambda(1-x)}{s^{-\rho}(1+\omega s^\alpha)^\gamma  } <1$, applying inverse of Sawi transform on both sides of the above equation and we will get the desired result
\begin{equation*}
\psi(x,t)=\sum_{n=0}^{\infty}(-\lambda)^n(1-x)^n t^{\rho n}E_{\alpha,\rho n+1}^{\gamma n}(-\omega t^\alpha)
\end{equation*}
\end{proof}

\begin{thm}
This theorem presents the solution for the Cauchy-type fractional differential equation.
\begin{equation}\label{hilfer-prabhakar prob with Pr integral}
\mathcal{D}_{\alpha,\omega,0^+}^{\gamma,\rho,\nu}
\psi(t)
=\lambda\mathcal{I}_{\alpha,\rho,\omega,0^+}^{\delta}\psi(t)+y(t),
\end{equation}
\begin{equation}\label{condition of M}
\left(\mathcal{I}_{\alpha,(1-\nu)(1-\rho),\omega,0^+}^{-\gamma(1-\nu)}\psi(t) \right)|_{t=0}= M
\end{equation}
\end{thm}
where~  $\psi(t)\in L_1[0,\infty):\rho\in(0,1),\nu\in[0,1]:\omega,\lambda \in \mathbb{C} : t,\alpha>0,K,\gamma,\delta\geq 0,
~~is~~ given~~ by$
\begin{equation}
\begin{split}
\psi(t)
=  \sum_{n=0}^{\infty}\lambda^n \mathcal{I}_{\alpha,\rho(2n+1),\omega,0^+}^{\gamma+n(\delta+\gamma)}y(t) +M  \sum_{n=0}^{\infty}\lambda ^n& t^{\rho(2n+1)+\nu(1-\rho)-1} \\&
\times E_{\alpha,\nu(1-\rho)+\rho(2n+1)}^{\delta n +\gamma n +\gamma-\gamma \nu} (\omega t^\alpha)
\end{split}
\end{equation}
\begin{proof}The solution of the Cauchy-type fractional differential equation can be obtained by applying the Sawi transform to both sides of equation (\ref{hilfer-prabhakar prob with Pr integral}), and using the equations (\ref{Sawi of hilfer-prabhakar}), (\ref{condition of M}), (\ref{Sawi of prabhakar integral}) 
\begin{equation*}
\begin{split}
Sa[\lambda \mathcal{I}_{\alpha,\rho,\omega,0^+}^{\delta}\psi(t)+y(t)]&
=Sa\left[\lambda \mathcal{I}_{\alpha,\rho,\omega,0^+}^{\delta}\psi(t),s \right]+Sa[y(t),s]\\&
=\lambda Sa[(\psi *e_{\alpha,\rho,\omega}^\delta)(t),s ] +Sa[y(t),s] \\&
=\lambda Sa[\psi(t) t^{\rho-1}E_{\alpha,\rho}^{\delta}(\omega t^\alpha),s]+Sa[y(t),s]\\&
=\lambda s^2 \mathcal{T}(s)s^{\rho-2}(1-\omega s^\alpha)^{-\delta}+Sa[y(t),s]\\&
=\lambda s^\rho (1-\omega s^\alpha)^{-\delta} \mathcal{T}(s) + Sa[y(t),s],\\&
\end{split}
\end{equation*}

\begin{equation*}
\begin{split}
Sa\left[\mathcal{D} _{\alpha,\omega,0^+}^{\gamma,\rho,\nu}y(t),s \right]&
=  s^{-\rho} (1- \omega s^\alpha )^\gamma \mathcal{T}(s)  \\ &
- s^{\nu(1-\rho)-2} (1-\omega{s^\alpha} )^{\gamma \nu} \mathcal{I}_{\alpha,(1-\nu)(1-\rho),\omega,0^+}^{-\gamma(1-\nu)}\psi(t)|_{t=0},\\&
= s^{-\rho} (1- \omega s^\alpha )^\gamma \mathcal{T}(s)
- s^{\nu(1-\rho)-2} (1-\omega{s^\alpha} )^{\gamma \nu} M 
\end{split}
\end{equation*}
therefore
\begin{equation*}
\begin{split}
\mathcal{T}(s)&
=\frac{ Sa[y(t),s] +s^{\nu(1-\rho)-2} (1-\omega{s^\alpha} )^{\gamma \nu} M }
{ s^{-\rho} (1- \omega s^\alpha )^\gamma 
\left[1- \frac{\lambda s^\rho (1-\omega s^\alpha )^{-\delta} 
} {s^{-\rho}(1-\omega s^\alpha)^\gamma  }
\right]}\\&
= \frac{Sa[y(t),s] +s^{\nu(1-\rho)-2} (1-\omega{s^\alpha} )^{\gamma \nu} M  }{s^{-\rho} (1- \omega s^\alpha )^\gamma  }\sum_{n=0}^{\infty} \left[ \frac{\lambda s^\rho (1-\omega s^\alpha )^{-\delta} 
} {s^{-\rho}(1-\omega s^\alpha)^\gamma  }
\right]^n \\&
= \left(Sa[y(t),s] +s^{\nu(1-\rho)-2} (1-\omega{s^\alpha} )^{\gamma \nu} M \right)\sum_{n=0}^{\infty}\lambda^n s^{2\rho n+\rho}(1-\omega s^\alpha)^{-\delta n-\gamma n-\gamma}\\&
=Sa[y(t),s]\sum_{n=0}^{\infty}\lambda^n s^{2\rho n+\rho}(1-\omega s^\alpha)^{-\delta n-\gamma n-\gamma}\\&+ M\sum_{n=0}^{\infty}\lambda^n s^{2\rho n+\rho+\nu(1-\rho)-2}(1-\omega s^\alpha)^{-\delta n-\gamma n-\gamma+\gamma \nu}
\end{split}
\end{equation*}
for $ \left[ \frac{\lambda s^\rho (1-\omega s^\alpha )^{-\delta} } {s^{-\rho}(1-\omega s^\alpha)^\gamma  }
\right]<1 $, and by taking the inverse of Sawi transform on both sides of the above equation, we will get the desired result
\begin{equation*}
\begin{split}
\psi(t)
=  \sum_{n=0}^{\infty}\lambda^n  \mathcal{I}_{\alpha,\rho(2n+1),\omega,0^+}^{\gamma+n(\delta+\gamma)}y(t) +M  \sum_{n=0}^{\infty}& \lambda ^nt^{\rho(2n+1)+\nu(1-\rho)-1} \\&
\times E_{\alpha,\nu(1-\rho)+\rho(2n+1)}^{\delta n +\gamma n +\gamma-\gamma \nu} (\omega t^\alpha)
\end{split}
\end{equation*}
\end{proof}

\section{Conclusion}
In this work, we first obtained the Sawi transform for the Hilfer-Prabhakar fractional derivatives and their regularized versions. Then, we demonstrated how to use these results to find the solution of Cauchy-type fractional differential equations involving Hilfer-Prabhakar fractional derivatives by combining the Sawi and Fourier transforms, which utilizes the three-parameter Mittag-Leffler function. The results show that the Sawi transform is a useful tool for solving fractional differential equations.

\end{document}